\documentclass[12pt,a4paper,reqno]{amsart}
\usepackage{xypic,a4wide,amsmath,amssymb,amsthm}
\usepackage[mathcal]{euscript}
\usepackage{mathrsfs}
\let\mathcal=\mathscr

\newtheorem{thm}{Theorem}[section]
\newtheorem{corol}[thm]{Corollary}
\newtheorem{lemma}[thm]{Lemma}
\newtheorem{prop}[thm]{Proposition}

\newtheorem{assumption}[thm]{Assumption}
\theoremstyle{remark}
\newtheorem{rem}[thm]{Remark}
\newtheorem{ex}[thm]{Example}
\newenvironment{remark}{\begin{rem}\rm}{\qee\end{rem}}

\newcommand{\End}{\mbox{\it End}\,}
\newcommand{\cA}{{\mathcal A}}
\newcommand{\cE}{{\mathcal E}}

\newcommand{\cO}{{\mathcal O}}

\newcommand{\ati}{{{\mathcal D}_{\cE}}}
\newcommand{\lati}[1]{{\Lambda^{#1}{{\mathcal D}_\cE^\ast}}}

\newcommand{\qee}{\mbox{\hspace{0.2mm}}\hfill$\triangle$}

\begin{document}\thispagestyle{empty}
\begin{center}
{\Large\bf Koszul complexes and  spectral sequences \\[5pt] associated  with  Lie algebroids} \\[10pt]
{\sc Ugo Bruzzo}$^{\ddag\P\S\star\sharp}$ and {\sc Vladimir  N. Rubtsov}$^{\flat\natural\star}$  \\[5pt] {\small $\ddag$  SISSA (Scuola Internazionale Superiore di Studi Avanzati), \\ Via Bonomea 265, 34136 Trieste, Italy;  \\
$\P$ Departamento de Matem\'atica, Universidade Federal da \\ 
Para\'iba, Campus I, Jo\~ao Pessoa, PB, Brazil; \\
$\S$ INFN (Istituto Nazionale di Fisica Nucleare), Sezione di Trieste;\\ 
$\star$ IGAP (Institute for Geometry and Physics), Trieste; \\ 
$\sharp$ Arnold-Regge Institute for Algebra,  \\ Geometry and Theoretical Physics, Torino  
}\\[5pt]
  {\small 
$\flat$ Universit\'e d'Angers, D\'epartement de Math\'ematiques, \\  UFR Sciences, LAREMA, UMR 6093 du CNRS, \\
2 bd.~Lavoisier, 49045 Angers Cedex 01, France; \\
$\natural$ ITEP Theoretical Division, 25 Bol.~Tcheremushkinskaya,\\ 117259, Moscow, Russia
\\[5pt]  E-Mail: {\tt bruzzo@sissa.it}, {\tt Volodya.Roubtsov@univ-angers.fr}}
\end{center} 

\bigskip
\begin{quote}
\footnotesize  {\sc Abstract.} We study some spectral sequences associated with a locally free $\cO_X$-module $\cA$   which has a Lie algebroid structure. 
Here $X$ is either a complex manifold  or a regular scheme over an algebraically closed field $k$.  One  spectral sequence can be associated with $\cA$ by choosing a global section $V$ of $\cA$, and considering a  Koszul complex with a differential given by inner product by $V$. This spectral sequence is shown to degenerate at the second page by using Deligne's degeneracy criterion.

Another  spectral sequence we study arises when considering the Atiyah algebroid  $\ati$ of a holomolorphic vector bundle $\cE$ on a complex manifold.  If  $V$ is a differential operator on $\cE$ with scalar symbol, i.e, a global section of $\ati$, we associate with the pair $(\cE,V)$ a twisted Koszul complex.   The first 
spectral sequence associated with this complex is known to degenerate at the first page in the untwisted ($\cE=0$) case. \end{quote}

 \vfill
 \begin{quote}
\begin{center} \em \small This is a contributed paper to the proceedings of  the 2nd Workshop of the
S\~ao Paulo Journal of Mathematical Sciences:
``Jean-Louis Koszul in S\~ao Paulo, His Work and Legacy.''
Institute of Mathematics and Statistics, University of S\~ao Paulo, Brazil, November 13-14, 2019
\end{center}
\end {quote}
 \vfill
\noindent\parbox{.75\textwidth}{\hrulefill}\par
\noindent\begin{minipage}[c]{\textwidth} \footnotesize 
 \noindent   
Date: \today  \par\noindent
The authors gratefully acknowledge
financial support and hospitality during  visits to
Universit\'e d'Angers and {\sc sissa}. Support for this work was provided by {\sc prin} ``Geometria delle variet\`a algebriche,''     the {\sc infn} project {\sc gast} ``Gauge and string theories", the {\sc einstein} Italo-Russian project  ``Integrability in topological string and field theory,' and the {\sc geanpyl-ii}  Angers-{\sc sissa} project
``Generalized Lie algebroid strucures.''
\end{minipage}
\newpage
\noindent

\section{Introduction}
In this paper we consider some spectral sequences that one can attach to a Lie algebroid. To be more precise, if $X$ is a  complex manifold, or a regular noetherian  scheme over an algebraically closed field $k$ of characteristic zero, we consider a locally free $\cO_X$-module $\cA$ having a Lie algebroid structure (definitions will be given in the next Section). One can introduce a complex $\Omega_\cA^\bullet=\Lambda^\bullet\cA^\ast$
which is a generalization of the (holomorphic) de Rham complex $\Omega_X^\bullet$. Now a Lie algebroid $\cA$ comes with a morphism of sheaves of Lie $k$-algebras (the anchor morphism) to the tangent sheaf $\Theta_X$, and the kernel of the anchor is a sheaf of ideals of $\cA$ (and a sheaf of Lie $\cO_X$-algebras); this allows one to introduce, in analogy with the Hochschild-Serre spectral sequence \cite{Hoch-Serre53}, a filtration leading to a spectral sequence which converges to the hypercohomology $\mathbb H(X,\Omega_\cA^\bullet)$. This was already considered in  \cite{BKTV09} in the $C^\infty$ case; moreover, \cite{Rub-thesis,Roub80} describe this spectral sequence
in the case of the Atiyah algebroid of a vector bundle. 
 In \cite{Ugo-derived} and \cite{BMRT} this and other spectral sequences were studied in detail.   Lie-Rinehart algebras can be regarded as special cases of Lie algebroids, so that we get a spectral sequence for Lie-Rinehart algebras: this generalizes the Hochschild-Serre spectral sequence for ideals in Lie algebras \cite{Hoch-Serre53}.

Other spectral sequences arise when we fix a section $V$ of $\cA$; this yields a complex of the Koszul type, which we call a Lie-Koszul complex. Then the general machinery of homological algebra \cite{EGA3-I,Verdier-ast} produces two spectral sequences. In Section \ref{KosDegSec},
by using Deligne's degeneracy criterion \cite{Deligne-deg}, we show that the second spectral sequence degenerates. The fact that this spectral sequence satisfies the condition of Deligne's criterion means that the Lie-Koszul complex of a (holomorphic) Lie algebroid is formal (it is isomorphic, in the derived category of coherent sheaves, with the complex formed by its cohomology sheaves).

To study the first spectral sequence of a Lie-Koszul complex,  we specialize to the case when $\cA$ is the Atiyah algebroid of a holomorphic vector bundle $\cE$ on a complex manifold $X$ (Section \ref{sectionAtiyah}). 
 Let us recall that  $\ati$ is  the bundle of first order differential operators on $\cE$ with scalar symbol. $\ati$ sits in
an exact sequence of sheaves of $\cO_X$-modules
\begin{equation} \label{atiyah} 0 \to  \End(\cE) \to \ati \xrightarrow{\sigma } \Theta_X \to 0
\end{equation} where   $\sigma$  is the symbol map. This spectral sequence relates to the twisted holomorphic equivariant cohomology we introduced in \cite{BR-hec}. For $\cE=0$ (i.e, in the case of the de Rham complex) this spectral sequence was studied by   Carrell and Lieberman \cite{CL-vf} and Bismut \cite{Bismut04} when $X$ is K\"ahler manifold. In that case the spectral sequence degenerates at the first page.

{\bf Acknowledgements.} We thank Paul Bressler, Tony Pantev, Jean-Claude Thomas and Pietro Tortella for useful discussions.
U.B.~thanks the organizers of the ``Jean-Louis Koszul in S\~ao Paulo, His Work and Legacy'' workshop for their kind invitation.

\bigskip
\section{Formality of the Lie-Koszul complexes}\label{KosDegSec}
We consider    a (holomorphic) Lie algebroid $\cA$, over $X$, the latter being a complex manifold, or a regular noetherian scheme over an algebraically closed field $k$. We choose a global section $V$ of $\cA$ and consider the morphism (inner product) $i_V\colon\mathcal K_\cA^\bullet\to 
\mathcal K_\cA^{\bullet+1}$, where $\mathcal K_\cA^p = \Omega_\cA^{-p}$, $p\le 0$. We shall call $(\mathcal K_\cA^\bullet,i_V)$
the {\em Lie-Koszul complex associated with the pair $(\cA,V)$.} This generalizes the Koszul complex $(\Omega_X^{-\bullet},i_V)$   associated with the complex of differential forms on $X$ with the differential given by the inner product by a (holomorphic) vector field $V$ on $X$.
This will be called the {\em de Rham-Koszul complex} associated with the vector field $V$.

By general principles  \cite{EGA3-I,Verdier-ast} we can associate two spectral sequences with this complex, both converging to the hypecohomology $\mathbb H(\mathcal K_\cA^\bullet,i_V)$. In general,  if $\mathfrak A$, $\mathfrak B$ are Abelian categories, denote by $D^+(\mathfrak A)$ the derived category of complexes of objects in $\mathfrak A$ bounded from below, and let 
 $F\colon D^+(\mathfrak A) \to \mathfrak B$ a cohomological 
 functor.\footnote{$F$ is said to be a cohomological functor if it maps every distinguished triangle
 $$ \xymatrix{ X \ar[r]^u & Y  \ar[r]^v & Z  \ar[r]^w & X[1] }$$ to a long exact sequence 
 $$ \xymatrix {R^iF(X) \ar[r]^{R^iF(u)} &  R^iF(Y) \ar[r]^{R^iF(v)} &  R^iF(Z)  \ar[r]^{R^iF(w)\ \ }  &  R^{i+1}F (X) }.$$}
 Let $\mathcal K$ be an object in $D^+(\mathfrak A)$. We recall from \cite{EGA3-I,Verdier-ast} that with these data   one can associate two spectral sequences, both functorial in $\mathcal K$, 
 and both  converging to $R^{\bullet}F(\mathcal K)$. The first two pages of the first spectral sequence are 
 $$I_1^{p,q}=R^qF(\mathcal K^p),\qquad I_2^{p,q} = H^p(R^qF(\mathcal K))$$
 and the differential $d_1$ coincides (perhaps up to a sign, depending on   conventions) with the differential of the complex
 $\mathcal K$.
 The second page of the second spectral sequence is 
 $$I\! I_2^{p,q} = R^pF(H^q(\mathcal K)).$$
 
 The degeneration of the the second spectral sequence may be studied by means of {\em Deligne's degeneracy criterion} \cite{Deligne-deg}.
 Let us state it in  generality.  We shall  replace the derived category $D^+(\mathfrak A)$   by the bounded derived category $D^b(\mathfrak A)$.

\begin{thm}[Deligne]  The following two conditions are equivalent:

(i) the   spectral sequence $I\! I_\bullet$ degenerates at its second page for every choice of the functor $F$;

(ii) $\mathcal K^\bullet$ is isomorphic to $\oplus_i  H^i(\mathcal K^\bullet)[-i]$ in $D^b(\mathfrak A)$.

\end{thm}

(In the language of homological algebra, the second condition is called {\em formality} of the complex $\mathcal K^\bullet$.) 

To apply Deligne's criterion to our case we take $\mathcal A = Coh(X)$,
$\mathcal B = {\mathbf K}(\mbox{Ab})$ (the category of complexes of Abelian groups) and for $F$ we take the global section functor $\Gamma$. The object we fix in $D^b(X)$ is the Lie-Koszul complex $(\mathcal  K_\cA^\bullet,i_{ V})$.  
We denote by $\mathcal H_\cA^\bullet$ the cohomology sheaves of the complex $(\mathcal  K_\cA^\bullet,i_{ V})$, and by $Y$ the scheme of zeroes of $V$. It is a closed, possibly nonreduced, subscheme (analytical subspace) of $X$. The sheaves $\mathcal H_\cA^\bullet$ are supported on $Y$.
Let $j \colon  Y \to X$ be the scheme-theoretic inclusion, or the inclusion as a morphism in the category of analytic spaces  (a closed immersion). The functor $j_\ast$
is right adjoint to $j^\ast$, so that there are morphisms $j^\star\colon \mathcal F\to j_\ast j^\ast\mathcal F$ for every coherent sheaf $\mathcal F$ on $X$. There is a commutative diagram
\begin{equation}\label{commu}\xymatrix{ \mathcal K_\cA^p \ar[r]^{j^\star}\ar[d]^{i_V} &
j_\ast j^\ast \mathcal K_\cA^p \ar[d]^0 \\
\mathcal K_\cA^{p+1} \ar[r]^{j^\star} & j_\ast j^\ast \mathcal K_\cA^{p+1} 
}\end{equation}
i.e., $j^\star$ is a morphism of complexes if we equip $j_\ast j^\ast \mathcal K_\cA^\bullet$ with the zero morphisms.
Finally, $\mathcal H_\cA^\bullet \simeq j_\ast j^\ast \mathcal K_\cA^\bullet$.
Now we have:

\begin{prop} The morphism of complexes 
$j^\star\colon (\mathcal K_\cA^{\bullet}, i_{V}) \to (j_\ast j^\ast\mathcal K_\cA^{\bullet},0)$ is a quasi-isomorphism.
\end{prop}
As a consequence, the objects $(\mathcal K_\cA^{\bullet}, i_{V})$ and $\bigoplus_i \mathcal H_\cA^i[-i]$ are isomorphic in the derived category $D^b(X)$. By Deligne's degeneracy criterion, we obtain that the spectral sequence  $I\! I_\bullet$ degenerates at the second page.

We can also say something about the hypercohomology $\mathbb H^\bullet(\mathcal K_\cA)$.
Let us denote by $\dim Y$ the dimension of the highest-dimensional component of $Y$. The proof of the following result goes as in the case of the de Rham-Koszul complex treated in \cite{CL-vf}, p.~306.

\begin{prop} $\mathbb H^m(\mathcal K_\cA^\bullet,i_V)=0$ for $m>\dim Y$.
\label{CL-vf}
\end{prop}
\begin{proof} Where $V\ne 0$ the Lie-Koszul complex is exact, so that the supports of the cohomology sheaves $\mathcal H_\cA^q$ are contained in $Y$; hence $I\! I_2^{p,q}=0$ for $p>\dim Y$. Moreover, $\mathcal H_\cA^q=0$ for $q>0$. Thus $I\! I_2^{p,q}=0$ for $p+q>\dim Y$. By standard homological arguments we get the thesis. 
\end{proof}
If $\dim Y=0,1$, this gives an easy proof of the degeneration of the second spectral sequence   at the second page, since $d_2\colon I\! I_2^{p,q} \to I\! I_2^{p+2,q-1} $
vanishes in that case. One also has 
$$ \mathbb H^m(\mathcal K^\bullet_\cA,i_V) \simeq \bigoplus_{p+q=m} H^p(X,\mathcal H^q_\cA).$$
When $\dim Y=0$, the second page of the spectral sequence is such that $I\! I_2^{p,q} =0$ if $p\ne 0$.
  
\bigskip
\section{A spectral sequence associated with Atiyah algebroids}\label{sectionAtiyah}
In this section we study the spectral sequence $I_\bullet$ in the special case when the Lie algebroid $\cA$ is the Atiyah
algebroid $\ati$ of a holomorphic vector bundle $\cE$ on a complex manifold $X$ (as in eq.~\eqref{atiyah}).

We fix once and for all a section
$V$  in $\Gamma(\ati)$. The pair $(\cE,V)$ is called
an {\em equivariant holomorphic vector bundle.} (``Equivariant'' refers to the fact that $V$ covers the infinitesimal action of the vector field $\sigma(V)$ on $X$.) We consider the associated  Lie-Koszul complex, i.e.,  the complex $(\mathcal K_\cE^\bullet,i_{V})$
where $\mathcal K_\cE^{p} = \Lambda^{-p}{{\mathcal D}_\cE^\ast}$
for $p\le 0$, and $\mathcal K_\cE^{p}=0$ for $p>0$.  This twisted Koszul complex, or, to be more precise, its Dolbeault resolution, is a building block of a ``twisted holomorphic equivariant cohomology" that we introduced in \cite{BR-hec} and for which we proved a localization formula that generalizes Carrell-Lieberman's \cite{CL,Carr}, Feng-Ma's \cite{Feng-Ma05} and Baum-Bott's \cite{BaumBott} formulas.

The spectral sequence $I_\bullet$ relates in this case to the double complex we introduced in \cite{BR-hec}. For $\cE=0$ this spectral sequence was studied by Carrell and Lieberman \cite{CL-vf} (see also Bismut \cite{Bismut04}) and in turn relates to K.~Liu's ``untwisted'' holomorphic equivariant cohomology
\cite{Liu}.  

We denote by $\Omega_X^{p,q}$ the sheaf
of differential forms of type $(p,q)$ on $X$, and  consider the complex 
\begin{equation}\label{deco}Q^k_{\cE} (X)= \bigoplus_{q-p=k}\Gamma\left[
\Lambda^{p}{{\mathcal D}_\cE^\ast}\otimes_{\cO_X}\Omega^{0,q}_X\right]\end{equation}
with the differential $\bar\partial_{\cE,V}=\bar\partial_\cE + i_{V}$,
where by $\bar\partial_\cE$ we collectively denote the Cauchy-Riemann operators of the bundles $\Lambda^{p}{{\mathcal D}_\cE^\ast}$.
We denote by $H^\bullet_{V}(X,\cE)$ the cohomology of this complex.  For $\cE=0$ this reduces to the cohomology introduced by K.~Liu \cite{Liu} (see also Carrell and Lieberman \cite{CL-vf} and Bismut \cite{Bismut04}.)

\begin{remark} If $V=0$ then $H^k_{V}(X,\cE)=\bigoplus_{q-p=k} H^q(X,\lati{p})$.\label{V=0}
\end{remark}

 \begin{prop} The cohomology $H^\bullet_{V}(X,\cE)$  is isomorphic to the hypercohomology  $\mathbb H^\bullet(\mathcal K_\cE^\bullet)$ of the complex
$(\mathcal K_\cE^\bullet,i_{V})$.\end{prop}
\begin{proof} 
The double complex  $\Lambda^{-\bullet}{{\mathcal D}_\cE^\ast}\otimes_{\cO_X}\Omega^{0,\bullet}_X$  is an acyclic resolution of the complex $\mathcal K_\cE^\bullet$, and
the  total complex of   $(\Lambda^{-\bullet}{{\mathcal D}_\cE^\ast}\otimes_{\cO_X}\Omega^{0,\bullet}_X,i_V,\bar\partial_\cE)$ coincides with $(\lati{\bullet},\bar\partial_{\cE,V})$. (This resolution is not made by coherent sheaves, but the argument works anyway, just going into the category of sheaves of Abelian groups.) 
\end{proof} 
 We denote
$$G_{p}^q= \bigoplus_{0\le p'\le -p} \Gamma\left[
\Lambda^{p'}{{\mathcal D}_\cE^\ast}\otimes_{\cO_X}\Omega^{0,q}_X\right]$$
with $p\le 0$,
so that $G_\bullet^q$ is a descending filtration of $Q^q_\cE(X)$. Note that
$$G_p^q = G_p\cap Q^q_\cE(X)\qquad\mbox{and}\qquad G_p^{p+q}/G_{p+1}^{p+q} = 
\Gamma\left[
\Lambda^{-p}{{\mathcal D}_\cE^\ast}\otimes_{\cO_X}\Omega^{0,q}_X\right].$$
This filtration of the complex $(Q_\cE^{(\bullet)}(X) ,\bar\partial_{\cE,V})$ defines a spectral sequence whose zeroth page is
$$E_0^{p,q} = G_p^{p+q}/G_{p+1}^{p+q} = \Gamma\left[
\Lambda^{-p}{{\mathcal D}_\cE^\ast}\otimes_{\cO_X}\Omega^{0,q}_X\right].$$
The spectral sequence converges to the cohomology  $H^\bullet_{V}(X,\cE)$.
The differential $d_0$ coincides with $\bar\partial_\cE$, as one easily checks. Therefore,
$$E_1^{p,q} = H^q(E_0^{p,\bullet}, d_0) = 
H^q(\Gamma[\Lambda^{-p}{{\mathcal D}_\cE^\ast}\otimes_{\cO_X}\Omega^{0,\bullet}_X],\bar\partial_\cE) \simeq H^q(X,\lati{-p}).$$

It is now easy to check that this spectral sequence coincides with $I_\bullet$.

Henceforth we assume that the zero locus $Y$ of $V$ is a complex submanifold of $X$. Therefore it makes sense to consider the complex \eqref{deco}  on $Y$; after letting $\tilde\cE=\cE_{\vert Y}$, we denote this new complex
$Q^\bullet_{\tilde\cE}(Y)$. Denoting by $j\colon\tilde Y\to X$ the embedding, we have the restriction morphism $j^\ast \colon Q^ \bullet_{\cE} (X)\to Q^\bullet_{\tilde\cE}( Y)$, which is a morphism of filtered complexes. We are going to show that, under some conditions, this is a quasi-isomorphism.

Note that there is an exact sequence
\begin{equation} 0 \to\mathcal D_{\tilde \cE}    \to  \ati_{\vert Y} \to N_{Y/X}\to 0\label{exactseq} \end{equation}
where $N_{Y/X}$ is the normal bundle to $Y$. 
Since $V$ is zero on $Y$, the commutator $\mathbb L_{V}(u) = [V,u]$ is well defined if $u\in  \ati_{\vert Y} $. This operator vanishes on $\mathcal D_{\tilde \cE}$, so it is well defined on $N_{Y/X} $. If it is injective, by composing with the projection $   \ati_{\vert Y} \to N_{Y/X}$ it yields an isomorphism,  thus splitting the sequence \eqref{exactseq}.

For clarity, we stress what we are assuming:

\begin{assumption} The zero locus $Y$ of $V$ is a complex submanifold on $X$, and the morphism $\mathbb L_{V}\colon N_{Y/X} \to \ati_{\vert Y} $ is injective.
\label{assu}\end{assumption}

 This      implies  the following preliminary result. Let $\tilde{\mathcal K}^\bullet_{\tilde{\cE}}$ be the complex of sheaves on $Y$
 $$\tilde{\mathcal K}^p_{\tilde{\cE}}=\Lambda^{-p}\mathcal D_{\tilde{\cE}}^\ast$$ 
 with the zero differential.
\begin{lemma} $j^\ast\mathcal H^p_\cE \simeq \tilde{\mathcal K}^p_{\tilde{\cE}}$. In particular, $ \mathcal H^p_\cE=0$ if $-p>\dim Y$.
\label{restri}\end{lemma}
\begin{proof} There is a naturally defined morphism $j^\ast\mathcal H^p_\cE \to \tilde{\mathcal K}^p_{\tilde{\cE}}$. We need to show that
this gives an isomorphism between the stalks of the two sheaves. Considering the exact sequence \eqref{atiyah} 
restricted to the stalks at a point $y\in Y$, it splits, and one has
\begin{align*}({\mathcal K}^p_{{\cE}})_y \simeq \bigoplus_{q+q'=-p} (\Omega^q_X)_y \otimes \Lambda^{q'} (\End(\cE))_y\,\\
(\tilde{\mathcal K}^p_{{\tilde \cE}})_y \simeq \bigoplus_{q+q'=-p} (\Omega^q_Y)_y \otimes \Lambda^{q'} (\End(\cE))_y\,,\end{align*}
Let $\tilde V$ be the vector field $\tilde V=\sigma(V)$. It vanishes on $Y$. Then one knows that the cohomology of
the complex $(\Omega^{-\bullet}_X,i_{\tilde V})$ restricted to $Y$ is isomorphic to the cohomology of the complex 
$(\Omega^{-\bullet}_Y,0)$ \cite{Bismut04}. This, together with the K\"unneth theorem, implies the result.
\end{proof}
  
 The following result generalizes to the twisted case Theorem 5.1 in \cite{Bismut04}. The proof goes as in  \cite{Bismut04}, but for clarity we report it here, adapted to the present situation, and with some more details.
  
  \begin{thm} Under the Assumption \ref{assu}, the restriction morphism $j^\ast \colon Q^ \bullet_{\cE} (X)\to Q^\bullet_{\tilde\cE}(Y)$
  is a quasi-isomorphism.
  \end{thm}
  \begin{proof} Let $\mathfrak U$ be an open cover of $X$, and consider the \v Cech-Koszul complex
$$C^{(k)}(X) = \bigoplus_{p+q=k}\check C^p(\mathfrak U,\mathcal K_\cE^q)$$
with differential $\tilde\delta=\delta+i_V$, where $\delta$ is the usual \v Cech differential. We define the descending filtration 
$$F_q = \bigoplus_{p'\ge p \atop q } \check C^{p'}(\mathfrak U,\mathcal K_\cE^q),\qquad F_p^q = F_p\cap C^{(q)}(X)$$
so that $F_{q+1}^p\subset F_q^p$, $$F_{q}^{p+q}/F_{q+1}^{p+q}=\check C^p(\mathfrak U,\mathcal K_\cE^q),$$
 and $$\tilde\delta(F_q^p) \subset F_{q+1}^{p+1} + F_{q}^{p+1} = F_{q}^{p+1}.$$
 Let $(E_\bullet(X),d_\bullet)$ be the ensuing spectral sequence. 
 The $d_0$ differential acting on the 0-th page coincides with $i_V$, so that the first page of the spectral sequence is
 $$E_1(X)^{p,q}= \check C^p(\mathfrak U,\mathcal H_\cE^q).$$
 The differential $d_1$ acting on this complex is the \v Cech differential. By Lemma \ref{restri}, we also have
 $$E_1(X)^{p,q}\simeq \check C^p(\tilde{\mathfrak U},\tilde {\mathcal K}_{\tilde \cE}^q)$$
 where $\tilde{\mathfrak U}$ is the  open cover   of $Y$ obtained by restricting the open sets of $\mathfrak U$ to $Y$.
 
 Consider now  the complex
 $$C^{(k)}(Y) = \bigoplus_{p+q=k}\check C^p(\tilde{\mathfrak U},\tilde{\mathcal K}^q_{\tilde{\cE}}).$$
 The resulting spectral sequence $E_\bullet(Y)$ has a vanishing $d_0$ differential, hence $E_1(Y)$ coincides with the $E_0$ page. The restriction morphism $j^\ast$ induces a morphism $j^\ast \colon E_1(X) \to E_1(Y)$. By the commutativity of the diagram \eqref{commu}, this is an isomorphism and commutes with the respective differentials (which are the \v Cech differentials of the respective \v Cech complexes). By \cite[Ch.~XV, Thm.~3.2]{CE56} the successive pages of the two spectral sequences are isomorphic, and the spectral sequences converge to the same group. Therefore, the complexes $C^{(\bullet)}(X)$ and $C^{(\bullet)}(Y)$ are quasi-isomorphic.
 
Via the standard \v Cech-Dolbeault spectral sequence, the cohomology of the complex $C^{(\bullet)}(Y)$ is, after taking a direct limit on the covers $\mathfrak U$, isomorphic to the cohomology of $(Q^{\bullet}_{\tilde{\cE}}(Y),\bar\partial_{\tilde\cE})$. In the same way, the cohomology of 
$C^{(\bullet)}(X)$ is isomorphic, after taking a direct limit, to the cohomology of
$(Q_\cE^{\bullet}(X),\bar\partial_{\cE,V})$. This concludes the proof.\end{proof}

  \begin{corol} $H^k_{V}(X,\cE)\simeq\bigoplus_{q-p=k}H^q(  Y,\Lambda^p\mathcal D_{\tilde \cE}^\ast)$.\label{corol}
  \end{corol}
  \noindent (Compare with Remark \ref{V=0}.)
  \begin{proof} Since $V=0$ on $Y$ this follows from Remark \ref{V=0}.\end{proof}
  
  Let us eventually consider the first spectral sequence $I_\bullet$. Its first page is 
  $$I_1^{p,q}=H^q(X,\Lambda^{-p}{{\mathcal D}_\cE^\ast}).$$ 
  In the untwisted ($\cE=0$) case, and assuming that $X$ is compact and K\"ahler, Carrell and Lieberman \cite{CL-vf}, by an argument inspired by Deligne's degeneracy criterion, show    that $d_1=0$, so that this spectral sequence degenerates at the first page.

         \bigskip
\frenchspacing
\def\cprime{$'$} \def\cprime{$'$} \def\cprime{$'$} \def\cprime{$'$}

\end{document}